\documentclass[12pt]{article}
\usepackage{amsmath,amssymb,amsthm}
\usepackage[margin=3cm]{geometry}
\usepackage{blindtext}
\usepackage{titlesec}
\title{Sections and Chapters}

\title{\bf  Homological propeties of Bimeasure algebras and their BSE properties}

\author{ Ali Rejali $^1$, \thanks{2020 Mathematics Subject Classifcation. Primary: 46J05; Secondary: 46J10} ,  \and Maryam Aghakoochaki  $^2$, \thanks{Corresponding author}}

\date{
	$^1$Department of Pure Mathematics, Faculty of Mathematics and Statistics, University of Isfahan, Isfahan 81746-73441, Iran \\ \texttt{rejali@sci.ui.ac.ir, Orcid: 0000-0001-7270-665X}\\%
	Isfahan University  \\ \texttt{mkoochaki@sci.ui.ac.ir, Orcid: 0000-0002-3851-6550}	$^2$\\[2ex]%
	\today
}

\newcommand{\tnorm}[1]{{\left\vert\kern-0.25ex\left\vert\kern-0.25ex\left\vert #1 
		\right\vert\kern-0.25ex\right\vert\kern-0.25ex\right\vert}}

\theoremstyle{plain}
\newtheorem{thm}{Theorem}[section]

\newtheorem{cor}[thm]{Corollary}
\newtheorem{lem}[thm]{Lemma}
\theoremstyle{definition}

\newtheorem{DEF}[thm]{\bf Definition}

\begin{document}
	\maketitle

\begin{abstract}
 Let $G$ and $H$ be locally compact  groups. $BM(G, H)$ denoted  the Banach algebra
of bounded bilinear forms on $C_{0}(G)\times C_{0}(H)$.In this paper, the homological properties of Bimeasure algebras are investigated. 
 It is found and approved that the  Bimeasure algebras $BM(G, H)$ is amenable if and only if $G$ and $H$ are discrete.
The correlation between the weak amenability of $BM(G, H)$ and $M(G\times H)$ is assessed. It is found and approved that the biprojectivity of the bimeasure algebra $BM(G, H)$ is equivalent to the finiteness of $G$ and $H$.
Furthermore, we show that the bimeasure group algebra $BM_{a}(G, H)$ is a BSE algebra.
 It will be concluded that $BM(G, H)$ is a BSE- algebra if and only if  $G$ and $H$ are discrete groups. 
 
\end{abstract}

\section{Introduction}\label{sec1}

                  Let $X$ and $Y$ be locally compact spaces. Set $V_{0}(X, Y)= C_{0}(X) \widehat\otimes C_{0}(Y)$,  the projective tensor product of $C_{0}(X)$ and $C_{0}(Y)$. Traditionally, the members of dual space of  $V_{0}(X, Y)$ are considered to be of bimeasures on
$X\times Y$.
For example, see  \cite{NV}.  Define
$$
BM(X, Y):= V_{0}^{*}(X, Y)
$$
The researcher in \cite{NV}, denoted the space of all bimeasures on $X\times Y$ by $BM(X, Y)$. Let $G$ and $H$ be locally compact groups. Then
$$
(C_{0}(G)\check\otimes C_{0}(H))^{*} = C_{0}(G\times H)^{*} = M(G\times H)
$$
with convolution product is a Banach algebra; see \cite{HR}.
Throughout this paper, $L_{1}(G)$ is the group algebra with convolution product which
is a closed subalgebra of measure algebra $ M(G) = {C_{0}(G)}^{*}$,
 for locally compact group $G$, where $C_{0}(G)$ is the Banach space of all bounded continuous functions on $G$ 
which are zero at infinity. Let $A$ and $B$ be Banach algebras. Then we denote the algebraic tensor product
by $A\otimes B$.  We denote the
injective tensor product by $A\check \otimes B$ and the projective tensor product by $A\hat\otimes B$.

   Some authors have studied the bimeasure algebras $BM(G, H)$, where $G$ and $H$ are Abelian groups. The researchers in \cite{GB}, introduced a multiplication and an adjoint operation on $BM(G, H)$ 
which turns it into a Banach algebra. Afterward, in \cite{GI}, they extended it to non-Abelian groups. Many of its algebraic and topological properties have been studied. For example due to 
$$
L^{1}(G\times H)= (L^{1}(G)\widehat\otimes L^{1}(H))^{*}
$$
introduced the group bimeasure algebra $BM_{a}(G, H)$ and discussed its properties. They showed that  $BM_{a}(G, H)$  plays a role in 
$BM(G, H)$ similar to that played by $L^{1}(G \times H)$ in $M( G \times H)$.
  
   In this paper, we discuss some other algebraic and topological properties of bimeasure algebras. We have shown that  $BM_{a}(G, H)$ is a $\textup {BSE}$- algebra. Moreover, $BM(G, H)$ is a $\textup {BSE}$- algebra if and only if the groups $G$ and 
$H$ are discrete. Also, many homological and cohomological properties of bimeasure algebras $BM(G, H)$ and $BM_{a}(G, H)$were determined. We show that whenever bimeasure algebras are amenable, weak amenable,  contractive, biprojective, and biflat.

The  Bochner-Schoenberg-Eberlein (BSE) is derived from the famous theorem proved in 1980 by Bochner and Schoenberg for the group of real numbers; \cite{SB} and \cite{P1}. The researcher in \cite{N2},  revealed that if $G$ is any locally compact abelian group, then the group algebra $L_{1}(G)$ is a BSE algebra.
 The researcher in \cite{P1},\cite{E7},\cite{E8} assessed the commutative Banach algebras that meet the Bochner-Schoenberg-Eberlein- type theorem and explained their properties.
The authors in \cite{MAR} and \cite{MAR2},
investigated and assess the correlation between different types of BSE- Banach algebras
$A$, and the Banach algebras $C_{0}(X, A)$ and $L_{1}(G, A)$. Next, we in \cite{MAR3}, the
$\textup{BSE}$ and $\textup{BED}$- property of tensor Banach algebra where assesed. 


                           The basic terminologies and the related information on  $\textup{BSE}$-  algebras are extracted from \cite{E6}, \cite{E7}, and \cite{E8}.
Let $A$ be a  commutative semisimple Banach algebra, and $\Delta(A)$ be the character space of $ A$ with the Gelfand topology. In this study, $\Delta(A)$ represents the set of all non-zero multiplicative
linear functionals over $A$.
 Assume that $C_{b}(\Delta(A))$ is the space consisting of all complex-valued continuous and bounded functions on $\Delta(A)$.
A continuous linear operator $T$ on $A$ is named a multiplier if for all $x,y\in A$, $T(xy)=xT(y)$.
The set of all multipliers on $A$ will be expressed as $M(A)$. It is obvious that $M(A)$ is a Banach algebra, and if $A$ is an unital Banach algebra, then 
$M(A)\cong A$.  As observed in   \cite{klar},  for each $T\in M(A)$ there exists a unique bounded  continuous function $\widehat{T}$ on $\Delta(A)$ where the following is yield:
$$\varphi(Tx)=\widehat{T}(\varphi)\varphi(x),$$
for all $x\in A$ and $\varphi\in \Delta(A)$.
By setting $\{\widehat{T}: T\in M(A)\}$, the $\widehat{M(A)}$ is yield.

 If the Banach algebra $A$ is  semisimple, then the Gelfand map $\Gamma_{A}: A\to C_{0}(\Delta(A)) $ which $f\mapsto \hat f$, is injective. A bounded complex-valued continuous function $\sigma$ on $\Delta(A)$  is named a  BSE function, if there exists a positive real number $\beta$ in a sense that for every finite complex-number $c_{1},\cdots,c_{n}$,  and  the same many $\varphi_{1},\cdots,\varphi_{n}$ in $\Delta(A)$ the following inequality
$$\mid\sum_{j=1}^{n}c_{j}\sigma(\varphi_{j})\mid\leq \beta\|\sum_{j=1}^{n}c_{j}\varphi_{j}\|_{A^{*}}$$
holds.\\
The set of all BSE-functions is expressed by $C_{\textup{BSE}}(\Delta( A))$, where for 
 each $\sigma$, the BSE-norm of $\sigma$, $\|\sigma\|_{\textup{BSE}}$ 
is   the infimum of all  $\beta$s  applied in the above inequality.  The researcher in [\cite{E6}, Lemma1] proved that $(C_{\textup{BSE}}(\Delta(A)), \|.\|_{\textup{BSE}})$ is a semisimple Banach subalgebra of $C_{b}(\Delta(A))$. Algebra $A$ is named a BSE algebra  if it meets the following condition:
$$\widehat{M(A)}= C_{\textup{BSE}}(\Delta(A)).$$ 
If $A$ is unital, then $\widehat{M(A)} = \widehat{A}\mid_{\Delta(A)}$, indicating that $A$ is a  BSE algebra if and only if  $C_{\textup{BSE}}(\Delta(A)) = \widehat{A}\mid_{\Delta(A)}$.

     Throughout this paper, the group algebra $L^{1}(G)$ with convolution product is a Banach subalgebra of measure algebra $M(G)= {C_{0}(G)}^{*}$, for locally compat group $G$. Such that $C_{0}(G)$ is the space of bounded functions on $G$
continuous with a limit of zero at infinity. Let $A$ and $B$ be Banach algebras. Then we denote the algebraic tensor product
by $A\otimes B$.  We denote the
injective tensor product by $A\check \otimes B$ and the projective tensor product by $A\hat\otimes B$.

\section{Fundamental Theorem of Grothendieck}

                 Let $X$ and $Y$ be locally compact Hausdorff spaces and $T\in (C_{0}(X)\hat\otimes C_{0}(Y))^{*}$. There exist a regular Borel probability
measures $\mu$ on $X$ and $\nu$ on $Y$, the sequences $(h_{n})\in L^{2}(\mu)$ and $(k_{n})\in  L^{2}(\nu)$, such that 
$$
T(f\otimes g)= \sum_{n=1}^{\infty} <f, h_{n}>_{H}<g,k_{n}>_{K}
$$  
for all $f\in H$ and $g\in K$, where $H= L^{2}(\mu)$ and $K= L^{2}(\nu)$. On the other hand 
\begin{align*}
T(f\otimes g) &= \sum_{n=1}^{\infty} \int_{X} f(x)\overline{h_{n}(x)}d\mu(x)\int_{Y}g(y)\overline{k_{n}(y)}d\nu(y)\\
                      &= \int_{X}\int_{Y} f(x)g(y) \sum_{n=1}^{\infty}\overline{h_{n}(x)k_{n}(y)}d\mu(x)d\nu(y)\\
                      &= \int_{X\times Y} f\otimes g(x,y).\overline{w(x,y)}d\mu\otimes \nu(x,y) 
\end{align*}
for all $f\in H$ and $g\in K$, where 
$$w(x,y)= \sum_{n=1}^{\infty}\overline{h_{n}(x)k_{n}(y)}$$
 Then 
$$
T(h)= \int_{X\times Y} hwd\mu\otimes \nu
$$
for all $h\in C_{0}(X)\hat\otimes C_{0}(Y)$.

     Moreover, if $T\in (C_{0}(X)\check \otimes C_{0}(Y))^{*}$, then there exists $\eta\in M_{b}(X\times Y)$ such that 
$$
T(h)= \int_{X\times Y} hd\eta
$$
for all $h\in C_{0}(X)\check \otimes C_{0}(Y)= C_{0}(X\times Y)$.
\begin{lem}
(i)
Always,
$\|h\|_{\epsilon}\leq \|h\|_{\pi}$ for each $h\in C_{0}(X)\hat\otimes C_{0}(Y)$. Therefore $ C_{0}(X)\hat\otimes C_{0}(Y)\subseteq C_{0}(X)\check\otimes C_{0}(Y)$.\\
(ii)
$$(C_{0}(X)\check\otimes C_{0}(Y))^{*}\subseteq (C_{0}(X)\hat\otimes C_{0}(Y))^{*}.$$
(iii)
$$\overline{C_{0}(X)\hat\otimes C_{0}(Y)}^{\|.\|_{\epsilon}}= C_{0}(X)\check\otimes C_{0}(Y).$$
\end{lem}
\begin{proof}
(ii)
Let $ T \in (C_{0}(X)\check\otimes C_{0}(Y))^{*}$, so there exists $m>0$ such that 
$$
\mid T(h)\mid\leq m\|h\|_{\epsilon}\leq m\|h\|_{\pi}
$$
Then $T\in (C_{0}(X)\hat\otimes C_{0}(Y))^{*}$.\\
(iii) We have: 
$$ C_{0}(X)\otimes C_{0}(Y)\subseteq C_{0}(X)\hat\otimes C_{0}(Y)\subseteq C_{0}(X)\check\otimes C_{0}(Y).$$
 Then the following is the yield:
\begin{align*}
\overline{C_{0}(X)\otimes C_{0}(Y)}^{\|.\|_{\epsilon}} &= C_{0}(X)\check\otimes C_{0}(Y)\\
                                                                                           &\subseteq \overline{C_{0}(X)\hat\otimes C_{0}(Y)}^{\|.\|_{\epsilon}}\\
                                                                                           &\subseteq \overline{C_{0}(X)\check\otimes C_{0}(Y)}^{\|.\|_{\epsilon}} \\
                                                                                              &\subseteq C_{0}(X)\check\otimes C_{0}(Y)
\end{align*}
Therefore, 
$$\overline{C_{0}(X)\hat\otimes C_{0}(Y)}^{\|.\|_{\epsilon}}= C_{0}(X)\check\otimes C_{0}(Y).$$
\end{proof}
\begin{DEF}
Let $X$ and $Y$ be locally compact Hausdorff spaces. The bimeasure algebra, $BM(X,Y)$ is the Banach space consisting of all measure such $\eta= w(\mu\otimes \nu)$ where
$\mu$ and $\nu$ are regular Borel probability
measures  on $X$ and on $Y$, respectively,  and $w\in L^{2}(\mu\otimes \nu)$.
\end{DEF}
\begin{thm}
Let $X$ and $Y$ be locally compact Hausdorff spaces. Then
$$
BM(X,Y)= (C_{0}(X)\hat\otimes C_{0}(Y))^{*}
$$ 
\end{thm}
\begin{proof}
Put 
\begin{align*}
G:~  (C_{0}(X)\hat\otimes C_{0}(Y))^{*} &\rightarrow BM(X,Y)\\
                                                                     & T\mapsto \eta_{T}
\end{align*}
where $\eta_{T}= w_{T}(\mu_{T}\otimes \nu_{T})$ such that $\mu_{T}$ and $\nu_{T}$ are regular Borel probability
measures  on $X$ and $Y$ and $w_{T}\in L^{2}(\mu_{T}\otimes \nu_{T})$, such that
$$
T(f\otimes g)= \int_{X}\int_{Y} w_{T}(x,y)f(x)g(y)d\mu_{T}(x)d\nu_{T}(y) 
$$
for all $f\in C_{0}(X)$ and $g\in C_{0}(Y)$. As observed in [\cite{GB}, Corrolary1.3],  there exists unique measures $\lambda_{X}, \lambda_{Y}$ such that 
$$
\mid T(f\otimes g)\mid\leq K_{G}^{2}\|f\|_{2}\|g\|_{2}
$$ 
for all $f\in C_{0}(X)$ and $g\in C_{0}(Y)$.  Therefore there exist unique measures $\mu_{T}$ and $\nu_{T}$ and $w_{T}$ corresponding to $ T$ where $w_{T}$ is in the Hilbert space $H\otimes K$. This implies that the map $T \mapsto \eta_{T}$ is
injective and well-defined. If $\mu\in P(X)$, $\nu\in P(Y)$ and $w\in L^{2}(\mu\otimes \nu)$, then 
$$
T(f, g)= \int_{X}\int_{Y} w(x, y)f(x)g(y)d\mu(x)d\nu(y)
$$
for $f\in C_{0}(X)$ and $g\in C_{0}(Y)$.  Thus $T$ is bilinear and 
$$
\mid T(f, g)\mid\leq \|f\|_{\infty}\|g\|_{\infty}\|w\|_{2, \mu\otimes\nu}.\|\mu\|_{2}\|\nu\|_{2}
$$
We have
\begin{align*}
\mid<w, 1>\mid &= \mid\int_{X\times Y}w^{2}(x, y)d\mu\otimes\nu(x,y)\mid^{\frac{1}{2}}.\mid\int 1(x,y)d\mu\otimes\nu(x,y)\mid^{\frac{1}{2}}\\
                          &=  \mid\int_{X}\int_{Y}w^{2}(x, y)d\mu(x)d\nu(y)\mid^{\frac{1}{2}}\|\mu\|.\|\nu\|
\end{align*}
As a result
$$
\mid T(f, g)\mid\leq \|f\|_{\infty}\|g\|_{\infty}\|w\|_{2}
$$
Then $T$ is a continuous bilinear function. So it has a continuous extension $\overline{T}\in (C_{0}(X)\hat\otimes C_{0}(Y))^{*}$, such that $G(\overline {T})= \eta$. Because
$$
G(\overline {T})(f\otimes g)= \int_{X}\int_{Y} w_{\overline{T}}(x,y)f(x)g(y)d\mu_{\overline{T}}(x)d\nu_{\overline{T}}(y)
$$
and 
$$
{{T}}(f\otimes g)= \int_{X}\int_{Y} w(x, y)f(x)g(y)d\mu(x)d\nu(y)
$$
for $f\in C_{0}(X)$ and $g\in C_{0}(Y)$. Thus 
$$
\mid G(\overline {T})(f\otimes g)\mid\leq \|w\|_{2}.\|f\|_{2}\|g\|_{2}
$$
and 
$$
\mid {{T}}(f\otimes g)\mid\leq K_{G}^{2}.\|f\|_{2}\|g\|_{2}
$$
Therefore $\eta_{T}= \eta$ where $\eta= w (\mu\otimes\nu)$ and $\eta_{T}= w_{T}( \mu_{T}\otimes\nu_{T})$. Hence $G$ is continuous, because
$$
\mid G( {T})(f\otimes g)\mid\leq \|w_{T}\|_{2}.\|f\|_{2}\|g\|_{2}
$$
This implies that
$$
\|G(T)\|\leq \|w_{T}\|_{2}
$$
and 
$$
{\overline T}(L)= <L, w_{T}>, \quad \|\overline T\|= \|w_{T}\|
$$
for $L\in H\hat\otimes K = (H\otimes K)^{*}$. Therefore $G$ is an isomorphism and 
$$
\|G(T)\|= \|T\|= \|w_{T}\|.
$$
See [\cite{wr}, Chapter 4].
\end{proof}

	\section{The BSE property for $BM_{a}(G, H)$ }	
Let $BM_{a}(G, H)$ be the closure of  $L^{1}(G\times H)$ in $BM(G, H)$ . The   researcher in \cite{GB}, revealed that $BM_{a}(G, H)$ plays a
role in $BM(G, H)$ similar to that played by $L^{1}(G\times H)$ in $M(G\times H)$. Assume that $f\in L^{1}(G\times H)$, so $u_{f}$ is a bimeasure determined by following
$$
\mu_{f}(E)=\int_{G}\int_{H}f(x,y)\boldsymbol{\chi}_{E}(x,y)dxdy
$$
The  related information on $BM_{a}(G, H)$ are extracted from \cite{GB} as the following:\\
$BM_{a}(G, H)$ is a closed subalgebra and ideal  of  $BM(G, H)$ where Gelfand transform is the Fourier transform  and $\Delta(BM_{a}(G, H))= \hat G\times\hat H$.
It is approved that ${\cal M}(BM_{a}(G, H))= BM(G, H)$.
By using the definiton of
BSE-functions, the folowing result is immediate.
\begin{lem}\label{r2}
Let $A$ and $B$ be Banach algebras. If $\Delta(A)= \Delta(B)$ and $\|.\|_{A^{*}}\cong \|.\|_{B^{*}}$,  then $C_{BSE}(\Delta(A))= C_{BSE}(\Delta(B))$.
\end{lem}
\begin{cor}\label{cbse}
Let $G$ and $H$ be LCA groups. Assume that $G$ or $H$ is a discrete group, then 
$$
C_{BSE}(\Delta(BM(G, H))=C_{BSE}(\Delta(M(G\times H))
$$
\end{cor}
\begin{proof}
Put $A:= M(G\times H)$ and $B:= BM(G, H)$. According to [\cite{GB}, Theorem5.12] $\overline{A}= B$, so $A^{*}= B^{*}$ and due to[\cite{GB}, Theorem 6.2]  $\Delta(A)= \Delta(B)$. Therefore by
applying  Lemma \ref{r2}, $C_{BSE}(\Delta(BM(G,A))=C_{BSE}(\Delta(M(G\times H))$.
\end{proof}
 \begin{thm}
Let $G$ and $H$ be the LCA groups. Then $BM_{a}(G,H)$ is a $\textup{BSE}$- algebra.
\end{thm}          
\begin{proof}
Because $BM_{a}(G, H)$ has an approximate identity  by applying [\cite{E6}, Corollary5]  the following is yield:
$${{{\cal M}}(BM_{a}(G, H)\widehat)}\subseteq C_{\textup{BSE}}(\Delta(BM_{a}(G, H))).$$
On the other hand, according to [\cite{GI}, Theorem4.5], ${{{\cal M}}(BM_{a}(G, H)\widehat)}= BM(G, H)$ and by [\cite{GB}, Remark 4.2], $\Delta(BM_{a}(G, H)) = \hat G\times\hat H$.
Let $\chi_{1}\in \hat G$ and $\chi_{2}\in \hat H$, put
$$
f_{\chi_{1},\chi_{2}}: bG\times bH\to \mathbb C
$$
Defined by 
$$
f_{\chi_{1},\chi_{2}}(\gamma_{1}, \gamma_{2})= \gamma_{1}(\chi_{1})\gamma_{2}(\chi_{2})
$$
Where $bG = \Delta(Ap(G))$ is Bohr compactification of $G$. Due to the \cite{Kn} the following is the yield:
$$
f_{\chi_{1},\chi_{2}}\in C(bG\times bH)= C(bG)\check\otimes C(bH)
$$
In \cite{Kn}, showed that $A_{p}(G)$ is a commutative unital $C^{*}$- algebra of all complex-valued periodic.
If $(\gamma_{\alpha}^{1}, \gamma_{\alpha}^{2})\overset{w^{*}}{\to}(\gamma^{1}, \gamma^{2})$ in $Ap(G)^{*}$, then $\gamma_{\alpha}^{1}(f)\to \gamma^{1}(f)$ and $ \gamma_{\alpha}^{2}(g)\to \gamma^{2}(g)$ for all $f\in Ap(G)$
and $g\in A_{p}(H)$. In the special case, $<\hat G>\subseteq Ap(G)$ and $\gamma_{\alpha}^{1}(\chi_{1})\to \gamma^{1}(\chi_{1})$ and $ \gamma_{\alpha}^{2}(\chi_{2})\to \gamma^{2}(\chi_{2})$. Therefore
 $f_{\chi_{1},\chi_{2}}(\gamma_{\alpha}^{1}, \gamma_{\alpha}^{2})\to f_{\chi_{1},\chi_{2}}(\gamma_{1}, \gamma_{2})$. Then $f_{\chi_{1},\chi_{2}}$ is continuous function on $bG\times bH$ and
$$
\mid f_{\chi_{1},\chi_{2}}(\gamma_{1}, \gamma_{2})\mid = \mid \chi_{1}(\gamma_{1})\mid.\mid \chi_{2}(\gamma_{2})\mid \leq 1
$$
Therefore $f_{\chi_{1}\times \chi_{2}}\in C(bG\times bH)$. Define
$$
T: \langle f_{\chi_{1},\chi_{2}}\mid \chi_{1}\in\hat G, \chi_{2}\in\hat H\rangle \to \mathbb C
$$
by 
$$
T(\sum_{i=1}^{n} c_{i}f_{\chi_{i}^{1},\chi_{i}^{2}}):= \sum_{i=1}^{n} c_{i}\sigma(\chi_{i}^{1}\otimes \chi_{i}^{2})
$$
Where 
$$
\mid T(\sum_{i=1}^{n} c_{i}f_{\chi_{i}^{1},\chi_{i}^{2}})\mid \leq \|\sigma\|_{BSE}\|\sum_{i=1}^{n} c_{i}(\chi_{i}^{1}\otimes \chi_{i}^{2})\|
$$
This implies that $T$ is a well-defined and continuous map. Thus 
$$T\in C(bG\times bH)^{*}= M(bG\times bH)$$
 so there exists some $ \mu\in  M(bG\times bH)$ such that 
$$T(h)= \int_{bG\times bH}hd\mu$$
 for all $h\in C(bG\times bH)$.
In this special case, put $h= f_{\chi_{1},\chi_{2}}$, the following is yield:
$$
\sigma(\chi_{1}\otimes \chi_{2})= \int_{bG\times bH}f_{\chi_{1},\chi_{2}}d\mu(\gamma_{1},\gamma_{2})
$$
Put
\begin{align*}
F_{1}(\chi_{1}\otimes \chi_{2}):=\int_{bG\times bH}\gamma_{1}(\chi_{1})\gamma_{2}d\mu^{+}(\gamma_{1},\gamma_{2})\\
F_{2}(\chi_{1}\otimes \chi_{2}):=\int_{bG\times bH}\gamma_{1}(\chi_{1})\gamma_{2}d\mu^{-}(\gamma_{1},\gamma_{2})
\end{align*}
Where $\mu^{+}$ and $\mu^{-}$ are Hahn decomposition for $\mu$ such that $\mu = \mu^{+}-\mu^{-}$ and $\mu^{+}\bot \mu^{-}$. $F_{1}$ and $F_{2}$ are positive definite functons. The following is the yield
\begin{align*}
\sum_{i,j=1}^{n} c_{i}\overline{c_{j}} F_{1}(\chi_{i}^{1}\otimes\chi_{j}^{2})(\overline{\chi_{j}^{1}}\otimes\overline{ \chi_{j}^{2}}) &= \sum_{i,j=1}^{n} c_{i}\overline{c_{j}} F_{1}(\chi_{i}^{1}\overline{\chi_{j}^{1}}\otimes\chi_{j}^{2}\overline{ \chi_{j}^{2}})\\
                                                                                                                                                                                                                                        &= \int_{bG\times bH}\sum_{i,j=1}^{n} c_{i}\overline{c_{j}}\gamma_{1}(\chi_{i}^{1}\overline{\chi_{j}^{1}})\gamma_{2}(\chi_{i}^{2}\overline{\chi_{j}^{2}})d\mu^{+}(\gamma_{1},\gamma_{2})
\end{align*}
Where $\gamma_{1}\in bG= \Delta(A_{p}(G))$ and $\hat{G}\subseteq A_{p}(G)$. Thus
\begin{align*}
\sum_{i,j=1}^{n} c_{i}\overline{c_{j}} F_{1}(\chi_{i}^{1}\otimes\chi_{j}^{2})(\overline{\chi_{j}^{1}}\otimes\overline{ \chi_{j}^{2}})& = \int_{bG\times bH}\mid\sum_{i,j=1}^{n} c_{i}\overline{c_{j}}\gamma_{1}\otimes \gamma_{2}(\chi_{1}\otimes\chi_{2}\mid^{2}
                                                                                                                                                                                                                                       &\geq 0
\end{align*}
We know that $(bG\widehat)= (\hat G, \tau_{d})$. So there exists $\mu_{1}$  such that
\begin{align*}
\mu_{1} &\in M(bG\times bH\widehat)\\
              & = M((bG\widehat)\times (bH\widehat))\\
               & = M(G\times H)
\end{align*}
such that 
\begin{align*}
F(\chi_{1}\otimes \chi_{2}) &= \int_{G\times H}\chi_{1}(x)\chi_{2}(y)d\mu_{1}(x, y)\\
                                              &= \int_{\hat G_{1}}\int_{\hat H_{1}}\gamma_{1}(\chi_{1})\gamma_{2}(\chi_{2})d\mu^{+}(x,y)    
\end{align*}
Where $G_{1}= (\hat{G}, \tau_{d})$ and $H_{1}= (\hat H, \tau_{d})$. Put $S= \hat{G}\times \hat{H}$, so by using Bochner theorem, there exist $\mu_{1}, \mu_{2}\in M(\hat{S})$ such that
$$
F_{1}(\chi_{1}, \chi_{2}) = \int_{G\times H}\chi_{1}(x)\chi_{2}(y)d\mu_{1}(x, y)
$$
and 
$$
F_{2}(\chi_{1}, \chi_{2}) = \int_{G\times H}\chi_{1}(x)\chi_{2}(y)d\mu_{2}(x, y)
$$
Set $\nu = \mu_{1}-\mu_{2}\in M(G\times H)$ such that 
$$
\hat\nu(\chi_{1}\otimes \chi_{2})= \sigma(\chi_{1}\otimes\chi_{2})
$$
So $\sigma = \hat\nu$, for $\nu\in BM(G, H)$. Therefore 
$$
 C_{\textup{BSE}}(\Delta(BM_{a}(G, H)) \subseteq{{{\cal M}}(BM_{a}(G, H)\widehat)}. 
$$
This completes the proof.
\end{proof}  


	\section{Homological properties of  bimeasure algebras}
In this section, we shall prove that the bimeasure algebra $BM(G, H)$  is a $\textup{BSE}$- algebra if and only if    $G$ and $H$ are discrete. Moreover, $BM(G, H)$ is amenable as a Banach algebra if and only if $G$ and $H$ are discrete and amenable
as a group. The researchers in \cite{GI},  defined convolution on  $BM(G, H)$.
Let $G$ be a  locally compact group and $M$ is an isometry of $L^{2}(G, \mu\star\nu)$ into $L^{2}(G\times G, \mu\times \nu)$ defined by 
$$
M\varphi(x_{1}, x_{2}) = \varphi(x_{1}x_{2})
$$
for $x_{1}, x_{2}\in G$, where $\mu$ and $\nu$ are Borel probability measures on $G$. Assume that $f\in C_{b}(G)$, set
$$
\check{f}(x)= f(x^{-1})   ~~~~ and ~~~~ f^{*}(x)= \overline{f(x^{-1})} \quad x\in G
$$
Let $\mu_{1}, \mu_{2}\in BM(G, H)$. Put 
$$
\mu_{1}\star \mu_{2}(f, g)= \mu_{1}\otimes\mu_{2}(Mf, Mg)
$$
and 
\begin{align}\label{r6}
\check{\mu_{1}}(f, g)= \mu_{1}(\check{f}, \check{g}) \quad  \tilde{\mu_{1}}(f, g)= \overline{\mu_{1}(f^{*}, g^{*})}
\end{align}
where $f\in C_{0}(G)$ and $g\in C_{0}(H)$. The multiplication and the adjoint operation (\ref{r6}), define a *- algebra on $BM(G, H)$ which extends the *- algebra structure of $M(G\times H)$. If $\mu, \nu\in BM(G, H)$, then $\|\tilde\mu\|= \|\mu\|$ and
$$
\|\mu\star\nu\|\leq K_{G}^{2}\|\mu\|\|\nu\|.
$$
See [\cite{GI}, Theorem2.8].
The following terminologies and the related information are extracted from \cite{he}  and \cite{mer}.
Assume that $A$ is a Banach algebra and $E$ is a Banach $A$- bimodule. A bounded linear operator $D : A\to E$ is called a derivation if the following is the yield
$$
D(ab) = D(a)b+ aD(b) 
$$
for all $a,b \in A$.  $A$ is named a weakly amenable if  every continuous derivation  $D: A\to A^{*}$ is inner. Let $P$ be a Banach left $A$- module. Then $P$ is called projective(respectively, flat) if for every Banach left $A$- module
$E$ and $F$ , for each admissible epimorphism ( respectively, monomorphism) $T\in {}_ A\!{\cal B}(E, F)$ and all $V\in {}_ A\!{\cal B}(P, F)$ (respectively, $V^{'}\in {}_ A\!{\cal B}(E, P^{*})$), there exists $U\in {}_ A\!{\cal B}(P, E)$ 
(respectively, $U^{'}\in {}_ A\!{\cal B}(F, P^{*})$) such that
$$
ToU= V \quad ( respectively,~ U^{'}oT = V^{'})
$$
Where  ${}_ A\!{\cal B}(E, F)$ is the set of all bounded left $A$- module homomorphisms from $E$ into $F$. If $A$ is a projective (respectively, flat) as a Banach $A$- bimodule, then $A$ is named biprojective (respectively, biflat).
The Banach algebra $A$ is amenable if and only if $A$ biflat and it has bounded approximate identity; see \cite{D}.  In a special case, $BM(G,H)$ has an identity, so $BM(G,H)$ is biflat if and only if it is amenable.
\begin{thm}
Let $G$ and $H$ be the LCA  groups. Then $BM(G, H)$ is a $\textup{BSE}$- algebra if and only if    $G$ and $H$ are discrete.
\end{thm}
\begin{proof}
According to [\cite{GB}, Theorem5.12] $M(G\times H)$ is dense in  $BM(G, H)$ if and only if either $G$ or $H$ is a discrete group. 
Let  $G$ and $H$ be discrete groups.Then $G\times H$ is a discrete group, so $M(G\times H)$ is a BSE- algebra and $\overline{M(G\times H)}= BM(G,H)$. Due to Corrolary \ref{cbse}, the following is the yield:
$$
C_{BSE}(\Delta(BM(G,H)))= C_{BSE}(\Delta(M(G\times H))) = {BM(G,H\widehat)}
$$
This completes the proof.\\
Conversely, suppose, say, $G$ or $H$ is non-discrete, so $G\times H$ is non-discrete. Then there exists some $\mu\in M(G\times H)$ such that $\hat\mu\mid_{\Delta(M(G\times H))\setminus \hat G\times \hat H}\not\equiv 0$.
Therefore $\mu \in BM(G, H)$ and $\hat\mu\mid_{\Delta(BM(G, H))\setminus \hat G\times \hat H}\not\equiv 0$. Define 
\begin{align*}
f(\gamma)=
\begin{cases} 
  \check\mu(\gamma),  & \mbox{if }\gamma\in\hat G\times \hat H \\  0, & \mbox{if }\gamma\in \Delta(BM(G,H))\setminus \hat G\times \hat H
\end{cases}
\end{align*}
but $f\neq \hat v$ for all $v\in BM(G, H)$. It is a contradiction and so $G$ and $H$ are discrete.
\end{proof}
The algebra $BM(G, H)$ has the identity $E= \delta e_{1}\otimes \delta e_{2}$ such that 
$$
E(f\otimes g)= f( e_{1})g( e_{2})
$$
Where $e_{1}= 1_{G}$ and $e_{2}= 1_{H}$, $f\in C_{0}(G)$ and $g\in C_{0}(H)$.
\begin{thm}
Let $G$ and $H$ be the locally compact groups. Put  $A= BM(G,H)$ and   $B= M(G\times H)$. Then the following are equivalent:\\
(i) $A$ is amenable Banach algebra.\\
(ii) $B$ is amenable Banach algebra.\\
(iii) $G$ and $H$ are discrete and amenable.
\end{thm}
\begin{proof}
($i \to  ii$) Let $X$ be a $B$- bimodule and $D: B\to X^{*}$ be a continuous derivation. It is obvious that $X$ is a $A$- bimodule. Assume that $T\in BM(G,H)$. \\
Define $T^{'}\in M(G\times H)$ as the following: \\
If $h\in C_{0}(G)\check\otimes C_{0}(H)$, then  there exists some sequence $(h_{n})\subseteq C_{0}(G)\hat\otimes C_{0}(H)$ such that $h_{n}\to h$ in injective topology.Now,define 
$$
T^{'}(h):= \underset{n\to\infty}{lim} T(h_{n})
$$
At this stage, $X$ is a Banach A- bimodule with the products
$$
~T.x:= T^{'}.x, \\
 ~x.T:= x.T^{'}
$$
Assume that $D: A\to X^{*}$ is an inner derivation. So $D^{'}: A\to X^{*}$ where $D^{'}(T):= D(T^{'})$ is an inner derivation, since $A$ is amenable there is $f\in X^{*}$ such that $D^{'}= ad_{f}$.
It is obvious that 
$$
C_{0}(G)\hat\otimes C_{0}(H)\subseteq C_{0}(G)\check\otimes C_{0}(H)
$$
So if $T\in B$, then
$$
T: C_{0}(G)\check\otimes C_{0}(H)\to \mathbb C  
$$
Thus there is some $u\in A$ such that $u:= T\mid_{C_{0}(G)\hat\otimes C_{0}(H)}$, so $T= u^{'}$. The following is the yield:
$$
D(T)= D(u^{'})= D^{'}(u)= ad_{f}(u)= ad_{f}(u^{'})= ad_{f}(T)
$$
Consequently $D$ is an inner derivation and $M(G\times H)$ is amenable.\\
($ii\to iii$) Assume that $M(G\times H)$ is amenable. Due to the \cite{DG}, $G\times H$ is discrete and amenable, so $G$ and $H$ are discrete and amenable.\\
($iii\to i$)  Assume that  $G$ and $H$ are discrete and amenable groups. Consequently $M(G\times H)$ is amenable and by applying Theorem6.2 in \cite{GI}, $M(G\times H)$ is dense in $BM(G,H)$.
Because $M(G\times H)$ is amenable, then $BM(G,H)$ is amenable.
\end{proof}
\begin{cor}
$BM(G,H)$  is biflat if and only if $G$ and $H$ are discrete and amenable groups.
\end{cor}
\begin{thm}
Let $G$ and $H$ be the  locally compact groups. Put  $A= BM(G,H)$ and   $B= M(G\times H)$. Then the following are equivalent:\\
(i) $A$ is weakly amenable Banach algebra;\\
(ii)  $B$ weakly amenable Banach algebra;\\
(iii) $G$ and $H$ are discrete groups.
\end{thm}
\begin{proof}
($i\to iii$) If $A$  is weakly amenable, then the subalgebras $M(G)$ and $M(H)$ are  weakly amenable. Therefore $G$ and $H$ are discrete groups.\\
($iii\to i$) Due to \cite{mer}, since ${\overline B}= A$ and $B$  is weakly amenable, $A$ is weakly amenable.\\
($ii\leftrightarrow iii$) $B$ is weakly amenable Banach algebra if and only if $G\times H$ is discrete if and only if $G$ and $H$ are discrete. 
\end{proof}

\begin{thm}
Let $G$ and $H$ be the locally compact Abelian groups. Then the following are equivalent:\\
(i) $BM(G,H)$ is a contractive Banach algebra;\\
(ii) $M(G\times H)$  is a contractive Banach algebra;\\
(iii) $G$ and $H$ are finite groups.
\end{thm}

\begin{proof}
($i\to ii$) Due to \cite{rn}, if $BM(G,H)$ is a contractive Banach algebra, then it is finite-dimensional and so 
$$
M(G\times H)= BM(G,H)
$$
is a contractive Banach algebra.\\
($ii\to iii$) If $M(G\times H)$   is a contractive Banach algebra, then $M(G\times H)$ is an amenable Banach algebra, so the groups $G$ and $H$ are discrete and $ L_{1}(G\times H)$ is a contractive algebra. Therefore the groups $G$ and $H$ are finite groups. \\
($iii\to i$) If  $G$ and $H$ are finite groups, then 
$$
C_{0}(G)\hat\otimes C_{0}(H)= C_{0}(G)\check\otimes C_{0}(H)
$$
Then 
$$
 L_{1}(G\times H)= M(G\times H)= BM(G, H)
$$
 is a contractive algebra.
\end{proof}

\begin{thm}
Let $G$ and $H$ be the locally compact Abelian groups. Then the following are equivalent:\\
(i)  $BM(G,  H)$  is a biprojective Banach algebra;\\
(ii) $G$ and $H$  are finite groups.
\end{thm}
\begin{proof}
 If $BM(G, H)$   is a biprojective Banach algebra, since it is unital then it is a contractive Banach algebra. Thus it is a amenable Banach algebra, so  $G$ and $H$ are discrete groups, then $ L_{1}(G\times H)$ is a contractive Banach algebra, as a result
$G\times H$ is a finite group and so $G$ and $H$ are finite groups.

  Conversely,it is obvious.
\end{proof}

\begin{thm}
Let $G$ and $H$ be locally compact groups. Then:\\
(i) $BM_{a}(G, H)$ is a contractive Banach algebra if and only if  $G$ and $H$ are finite groups.\\
(ii)  $BM_{a}(G, H)$ is a biprojective Banach algebra if and only if  $G$ and $H$ are compact groups.\\
(iii) $BM_{a}(G, H)$ is an amenable Banach algebra if and only if  $G$ and $H$ are amenable groups.\\
(iv) $BM_{a}(G, H)$ is biflat Banach algebra if and only if  $G$ and $H$ are amenable groups.\\
(v) Always, $BM_{a}(G, H)$  is weakly amenable. 
\end{thm}
\begin{proof}
We know that 
$$
BM_{a}(G, H) := L^{1}(G)\check\otimes L^{1}(H)
$$
Since $BM_{a}(G, H)$ has bounded approximate identity, $BM_{a}(G, H)$ is biflat if and only if it is amenable. Also, $L^{1}(G\times H)$ is dense in $ BM_{a}(G, H)$, thus $BM_{a}(G, H)$ is weakly amenable.
Since $ L^{1}(G)$ and $ L^{1}(H)$ are closed subalgebra with bounded approximate identity, so the proof is complete.
\end{proof}



\end{document}